\documentclass[reqno,11pt]{amsart}

\usepackage{amsmath}
\usepackage{graphicx}
\usepackage[english]{babel}
\usepackage[usenames]{color}
\usepackage{dsfont}
\usepackage{subfigure}

\usepackage{ifpdf}
\ifpdf
\usepackage{hyperref}
\else
\usepackage[hypertex]{hyperref}
\hypersetup{backref, colorlinks=true, bookmarks}
\fi

\usepackage{times}
\usepackage[normalem]{ulem}

\newtheorem{theorem}{Theorem}[section]
\newtheorem{lemma}[theorem]{Lemma}
\newtheorem{definition}[theorem]{Definition}
\newtheorem{corollary}[theorem]{Corollary}
\theoremstyle{remark}
\newtheorem*{remark}{Remark}

\numberwithin{equation}{section}

\newcommand{\R}{\mathbb{R}}
\newcommand{\C}{\mathbb{C}}

\newcommand{\N}{\mathbb{N}}

\newcommand{\E}[1]{\mathbb{E}\left[#1\right]}

\newcommand{\ET}[1]{\mathbb{E}_{\Theta}\left[#1\right]}
\newcommand{\PT}[1]{\mathbb{P}_{\Theta}\left[#1\right]}

\newcommand{\ETA}[2]{\mathbb{E}_{\Theta}^{(#1)}\left[#2\right]}
\newcommand{\PTA}[2]{\mathbb{P}_{\Theta}^{(#1)}\left[#2\right]}
\newcommand{\pta}{\mathbb{P}_{\Theta}^{(A_n)}}
\newcommand{\slan}{\sum_{\la \vdash n}\frac{1}{z_\la}}
\newcommand{\sla}{\sum_{\la}\frac{1}{z_\la}}

\newcommand{\la}{\lambda}

\newcommand{\eps}{\epsilon}

\newcommand{\nth}[1]{[t^n]\left[ #1 \right]}
\newcommand{\set}[1]{\left\{#1\right\}}
\DeclareMathOperator{\one}{\mathds{1}}

\renewcommand{\Re}{\mathrm{Re}}
\renewcommand{\Im}{\mathrm{Im}}

\newcommand{\Sn}{\mathfrak{S}_n}

\begin{document}

\title[Functional central limit for the weighted probability measure]{Large cycles and a functional central limit theorem for generalized weighted random permutations.}
\date{\today}

%
%
%
%
%
\author[A. Nikeghbali]{Ashkan Nikeghbali}
\address{Institut f\"ur Mathematik\\ Universit\"at Z\"urich\\ Winterthurerstrasse 190\\ 8057-Z\"urich,
Switzerland} \email{ashkan.nikeghbali@math.uzh.ch}

\author[J. Storm]{Julia Storm}
\address{Institut f\"ur Mathematik\\ Universit\"at Z\"urich\\ Winterthurerstrasse 190\\ 8057-Z\"urich,
Switzerland} \email{julia.storm@math.uzh.ch}

\author[D. Zeindler]{Dirk Zeindler}
\address{Department f\"ur Mathematik\\ Sonderforschunsgebreich 701\\Universit\"at Bielefeld\\ Bielefeld, 33501
\\Deutschland} \email{zeindler@math.uni-bielfeld.de}

\begin{abstract}
The objects of our interest are the so-called $A$-permutations, which are permutations whose cycle length lie in a fixed set $A$. They have been extensively studied with respect to the uniform or the Ewens measure. In this paper, we extend some classical results to a more general weighted probability measure which is a natural extension of the Ewens measure and which in particular allows to consider sets $A_n$ depending on the degree $n$ of the permutation. By means of complex analysis arguments and under reasonable conditions on generating functions we study the asymptotic behaviour of classical statistics. More precisely, we generalize results concerning large cycles of random permutations by Vershik, Shmidt and Kingman, namely the weak convergence of the size ordered cycle length to a Poisson-Dirichlet distribution. Furthermore, we apply our tools to the cycle counts and obtain a Brownian motion central limit theorem which extends results by DeLaurentis, Pittel and Hansen.
\end{abstract}

\maketitle

\tableofcontents

\section{Introduction}
\label{sec_intro}

Permutations are classical objects that appear in many mathematics fields. 
A special class of permutations are the so-called $A$-permutations, where $A$ is a non-empty subset of $\N$. 
We call an element $\sigma$ of the symmetric group $\Sn$ an $A$-permutation 
if $\sigma$ can be written as a product of disjoint cycles whose cycle-lengths are all in $A$.
These permutations have been extensively studied over the past thirty  years, 
a long list of references can be found for instance in \cite{Ya07b}.
It is well-known that with respect to the uniform measure the behaviour of $A$-permutations is similar to those of the whole permutation group.
To give a single example, in \cite{Ya07a} it was proved that for $n \rightarrow \infty$ the cycle counts $C_m$ (the number of cycles of length $m$ of $\sigma$) 
converge in distribution for $m\in A$ to independent Poisson distributed random variables $Y_m$ with expectation $1/m$. However, in all previous publications about $A$-permutations, one has only investigated its behaviour under the uniform measure and with $A$ being independent of $n$. Here, we consider the following more general $A_n$-weighted measure.
\begin{definition}
\label{def_A-weighted_probabililty_measure}
Let $A_n\subset \set{1,\dots,n}$ and $\Theta = \left(\theta_m  \right)_{m\geq1}$ be given, with $\theta_m\geq0$ for every $m\geq 1$.
We  define the $A_n$-weighted measure of $\sigma\in \Sn$ as
\begin{align}
  \PTA{A_n}{\sigma}
  :=
  \frac{1}{h_n n!} \prod_{m=1}^{\ell(\la)} \theta_{\la_m} \one_{\set{ \la_m\in A_n}}
  \label{eq_PTA_with_partition}
\end{align}
with $h_n = h_n(A_n)$ a normalization constant with $h_0:=1$,  $\la = (\la_1,\dots,\la_\ell)$ the cycle-type of $\sigma$ and $\ell(\la) = \ell$ the length of $\la$ (see Section \ref{sec21}).
\end{definition}

Define furthermore
\begin{align}\label{eq:def_Dn}
D_n := \set{1,\dots,n} \setminus A_n 
\quad \text{ and } \quad
d_n:= 
\begin{cases}
 \max D_n  & \text{if }D_n\neq \emptyset,\\
 1 & \text{ otherwise }.
\end{cases}
\end{align}
We investigate the behaviour of the measure $\pta$ for $d_n = o(n)$, that is to say the cycle lengths not contained in $A_n$ grow slowly (the precise assumptions on $d_n$ can be found in Theorem~\ref{thm:total_cycles_asymp_with_restriction}).
This assumption is motivated by a model in \cite[Section~6]{KoNi09b} about mod-Poisson 
convergence for an analogue of the Erd\"os-Kac Theorem for polynomials over finite fields.
 
The uniform measure or the Ewens measure on $\Sn$ are special cases of the $A_n$-weighted measure, obtained by choosing $A_n = \N$ and $\theta_m \equiv 1$ or $\theta_m \equiv  \theta$.  
Both are classical probability measures and are well-studied, see for instance \cite{ABT02}.

For $A_n = \N$, one obtains the weighted measure on $\Sn$, which was recently investigated in \cite{BeUeVe11}, \cite{ErUe11}, \cite{MaNiZe11}, \cite{NiZe11} (see also the
extensive background bibliography therein). Our study extends the results in \cite{NiZe11} about the cycle counts and the total cycle number to $\pta$ and is based on similar argumentations as those in \cite{NiZe11}. 
Furthermore, we apply our methods to objects which have so far not been considered for the weighted measure on $\Sn$.
More precisely, in Section~\ref{sec:large_cycles} we show that the size ordered cycle lengths converge in law to a Poisson-Diriclet-distribution. This result agrees with those by Vershik and
Shmidt~\cite{ShVe77} and Kingman~\cite{Ki77}, who studied the same asymptotic behaviour with respect to the Ewens measure.
Furthermore, we consider in Section~\ref{sec:func_limit_theorem} the number of cycles in a permutation with lengths not exceeding $n^x$ 
 and show that this process converges, after proper normalisation, to a standard Brownian motion.
This extends the results by Delaurentis and Pittel \cite{DePi85} (uniform measure) and Hansen \cite{Ha90} (Ewens measure) to
 the weighted measure and the $A_n$-weighted measure on $\Sn$.
A great advantage of our argumentation is that it is much more flexible and one can obtain easily the behaviour under further restrictions, see Sections~\ref{sec:restriction-small} and~\ref{sec:restriction_even_odd}.

It is clear that the asymptotic behaviour of all random variables on the group $\Sn$ with respect to the measure $\mathbb{P}_{\Theta}^{(A_n)}$ strongly depend on the sequence $\Theta = \left(\theta_m  \right)_{m\geq1}$ and 
it is thus necessary to impose appropriate assumptions on this sequence. More precisely, we will argue with generating functions, meaning that assumptions are imposed on the function
\begin{align}
\label{eq:def_g_Theta}
  g(t) = g_\Theta(t)
  &:=
  \sum_{n=1}^\infty \frac{\theta_n}{n} t^n .
\end{align}
The link of $g_{\Theta}(t)$ and the generating series of $(h_n)_{n \geq 1}$ is the starting point of our study; for $A_n = \{1, ..., n\}$ it is given by the well-known relation
\begin{align*}
 \sum_{n=1}^{\infty}{h_n t^n} = \exp(g_{\Theta}(t)).
\end{align*}
For general sets $A$ their relation will be stated in Lemma~\ref{lem:generating_hn_A}. We will choose $g_{\Theta}(t)$ in a way that allows us to apply the method of singularity analysis, see Definition~\ref{def_function_class_F_alpha_r}. We give more details in Section~\ref{sec_sing_analaysis}, but a good description of the method of singularity analysis can be found for instance in the book \cite{FlSe09} by Flajolet and Sedgewick.
%

The paper is organized as follows.  
In Section~\ref{sec_comb_andgen_of_Sn} some well known facts about the symmetric group $\Sn$ are presented and generating functions are introduced. In particular, we will recall the \textit{cycle index theorem}, wich will be used in computations throughout the whole paper.
In Section~\ref{sec_sing_analaysis} we determine the setting of our study, meaning that we properly define the assumptions on the functions $g(t)$ under consideration. With complex analysis arguments we establish our main tool, Theorem~\ref{thm:total_cycles_asymp_with_restriction}, which enables us to investigate the large-$n$ behaviour of coefficients of relevant functions. As a direct consequence, we deduce the asymptotic behavior of the normalization constant $h_n$. In Section~\ref{sec_cycles} we apply our methods to compute the characteristic functions of the cycle counts and of the total cycle number and we deduce a central limit theorem and a even stronger convergence result, namely mod-Poisson convergence. Further we investigate the behaviour of the large cycles and show that their asymptotic behaviour with respect to our general measure is the same as with respect to the Ewens measure. Finally, Section~\ref{sec:func_limit_theorem} is devoted to a functional central limit theorem giving the weak 
convergence of a certain functional of the cycle counts to the Brownian motion.  


\section{Combinatorics of $\Sn$ and generating functions}
\label{sec_comb_andgen_of_Sn}

This section is devoted to some basic facts about the symmetric group $\Sn$, partitions and generating functions.  Further, a useful lemma which identifies averages over $\Sn$ with generating functions is recalled. 
We give only a short overview and refer to \cite{ABT02} and \cite{Mac95} for more details.

\subsection{The symmetric group}\label{sec21}

All probability measures and functions considered in this paper are invariant under conjugation and 
it is well known that the conjugation classes of $\Sn$ can be parametrized with partitions of $n$.
This can be seen as follows: Let $\sigma\in \Sn$ be an arbitrary permutation and write $\sigma = \sigma_1\cdots \sigma_\ell$ with $\sigma_i$ disjoint cycles of length $\la_i$.
Since disjoint cycles commute, we can assume that $\la_1\geq\la_2\geq\cdots \geq\la_\ell$.
We call the partition $\la=(\la_1,\la_2,\cdots,\la_\ell)$ the \emph{cycle-type} of $\sigma$ and $\ell = \ell(\la)$ its \emph{length}. Then two elements $\sigma,\tau\in \Sn$ are conjugate if and only if
$\sigma$ and $\tau$ have the same cycle-type. Further details can be found for instance in \cite{Mac95}.
For $\sigma\in \Sn$ with cycle-type $\la$ we define $C_m$, the number of cycles of size $m$, and $T$, the total cycle number as 
\begin{align}
\label{eq_def_Cm_lambda}
C_m :=  \#\set{i ;\la_i = m} 
\quad \text{ and } \quad
T := T_n := \sum_{m=1}^n C_m.
\end{align}
It will turn out that all expectations of interest have the form $\frac{1}{n!} \sum_{\sigma\in \Sn} u(\sigma)$ for a certain class function $u$.
Since $u$ is constant on conjugacy classes, it is more natural to sum over all conjugacy classes. This is subject of the following lemma.

\begin{lemma}
\label{lem:size_of_conj_classes}
Let  $u: \Sn \to \C$ be a class function. For $C_m = C_m(\la)$ as in \eqref{eq_def_Cm_lambda} and $\mathcal{C}_\la$ the conjugacy class corresponding to the partition $\la$ we have
\begin{align*}
  |\mathcal{C}_\la| = \frac{|\Sn|}{z_\la} 
\quad \text{ with } \quad
z_\la:=\prod_{m=1}^{n} m^{C_m}C_m!
\end{align*}
and
\begin{align*}
  \frac{1}{n!} \sum_{\sigma\in \Sn} u(\sigma)
  =
  \slan u(\mathcal{C}_\la).
\end{align*}
\end{lemma}

\subsection{Generating functions}

Given a sequence $(a_{n})_{n\in\N}$ of numbers, one can encode important information about this sequence into a formal power series called the generating series.
\begin{definition}
\label{def_gneranting_function}
Let $(a_n)_{n\in\N}$ be a sequence of complex numbers. We then define
the (ordinary) generating function of $(a_n)_{n\in\N}$ as the formal power series
    \begin{align*}
      G(t) = G(a_n,t) = \sum_{n=0}^\infty a_n t^n.
    \end{align*}
%
We define $\nth{G}$ to be the coefficient of $t^n$ of $G(t)$, that is $\nth{G} := a_n$. 
\end{definition}
The reason why generating functions are powerful is the possibility of recognizing them
without knowing the coefficients $a_n$ explicitly. In this case one can try to use tools from analysis to extract information about $a_{n}$, for large $n$, from the generating function. 

The following lemma goes back to Polya and is sometimes called \textit{cycle index theorem}. It links generating functions and averages over $\Sn$.

\begin{lemma}
\label{lem:cycle_index_theorem}
Let $(a_m)_{m\in\N}$ be a sequence of complex numbers. Then
\begin{align}
\label{eq_symm fkt}
\sla \left(\prod_{m=1}^{\ell(\la)} a_{\la_{m}}\right) t^{|\la|}
= \sla \left(\prod_{m=1}^{\infty} (a_m t^m)^{C_m}\right) 
=
\exp\left(\sum_{m=1}^{\infty}\frac{1}{m} a_m t^m\right)
\end{align}
with the same $z_\la$ as in Lemma~\ref{lem:size_of_conj_classes}.
If one of the sums in \eqref{eq_symm fkt} is absolutely convergent then so are the others.
\end{lemma}
\begin{proof}
The proof can be found in \cite{Mac95} or can be directly verified using the definitions of $z_\la$ and the exponential function.
The last statement follows from the dominated convergence theorem.
\end{proof}
The crucial tool for our study, the relation of $g_\Theta(t)$ to the generating function of $h_n(A)$ for $A\subset \N$, can immediately be deduced from the previous lemma.
\begin{lemma}
\label{lem:generating_hn_A}
Let $A\subset \N$ and $\Theta$ be given as in Definition \ref{def_A-weighted_probabililty_measure} and define $D:=\N \setminus A$. We then have as formal power series
\begin{align}
\label{eq:generating_hn_with_restrictions}
\sum_{n=0}^\infty
h_n (A)  t^n
=
\exp \left( g_\Theta(t) - L_D(t) \right),
\end{align}
where $g_\Theta(t)$ is given by \eqref{eq:def_g_Theta} and $L_D(t)$ is the formal power series
\begin{align*}
 L_D(t):= \sum_{m\in D} \frac{\theta_m}{m} t^m. 
\end{align*}
\end{lemma}
\begin{proof}
We combine the definition of $h_n$, Lemma~\ref{lem:size_of_conj_classes} and Lemma~\ref{lem:cycle_index_theorem}. We get
\begin{align*}
 h_n(A)
&= 
\frac{1}{n!}\sum_{\sigma\in \Sn} \prod_{m=1}^{\ell(\la)} \bigl( \theta_{\la_m}  \one_{\set{\la_m \in A}} \bigr)
=
\slan  \prod_{m=1}^{\ell(\la)} \bigl( \theta_{\la_m}  \one_{\set{\la_m \in A}} \bigr).
\end{align*}
It follows that
\begin{align*}
\sum_{n=1}^\infty
h_n (A)  t^n
= 
\exp \left( \sum_{m = 1} ^\infty \frac{\theta_m}{m}  \one_{\set{m \in A}} t^m \right) 
=
\exp\left( g_\Theta(t) - L_D(t) \right).
\end{align*}
\end{proof}

\section{Singularity analysis for increasing cycle lengths}
\label{sec_sing_analaysis}

The main goal of this section is to provide by means of complex analysis arguments a tool, Theorem~\ref{thm:total_cycles_asymp_with_restriction}, that allows us to compute the large-$n$ behaviour of $h_n(A_n)$ and of other related quantities.

For this purpose, as mentioned in the introduction, we have to impose assumptions on the sequence $\Theta = (\theta_m)_{m\in\N}$. In view of Lemma~\ref{lem:generating_hn_A}, it is natural to impose them on the function
\begin{align}\label{eq:def_g_Theta_2}
   g(t) = g_\Theta(t)
  &:=
  \sum_{m=1}^\infty \frac{\theta_m}{m} t^m.
\end{align}
We shall apply  the method of singularity analysis to the function $g(t)$. 
A detailed description of singularity analysis can be found for instance in \cite[Section~VI]{FlSe09}. First we need a preliminary definition.

\begin{definition}
\label{def_delta_0}
Let $0< r < R$ and $0 < \phi <\frac{\pi}{2}$ be given. We then define
\begin{align}
\Delta_0 = \Delta_0(r,R,\phi) = \set{ z\in \C ; |z|<R, z \neq r ,|\arg(z-r)|>\phi}
\label{eq_def_delta_0}.
\end{align}

\begin{figure}[ht!]
\centering
 \includegraphics[width=.4\textwidth]{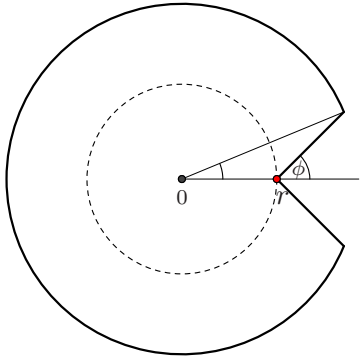}
\put(-74,62){\mbox{\scriptsize$0$}}
\put(-36.5,62){\mbox{$r$}}
\put(-30.5,73){\mbox{\scriptsize$\phi$}}

\caption{Illustration of $\Delta_0$}
\label{fig_delta_0}
\end{figure}

\end{definition}
Now we can introduce the family of functions we are interested in.
\begin{definition}
 \label{def_function_class_F_alpha_r}
Let $r,\vartheta>0$ and $K\in\R$ be given. We write $\mathcal{F}(r,\vartheta,K)$ for the set of all functions $g(t)$ satisfying
\begin{enumerate}
 \item $g(t)$ is holomorphic in $\Delta_0(r,R,\phi)$ for some $R>r$ and $0 < \phi <\frac{\pi}{2}$ and
 \item 
   \begin{align}
   g(t) = \vartheta \log\left( \frac{1}{1-t/r} \right) + K + O \left( t-r \right)   \text{ as } t\to r.
   \label{eq_class_F_r_alpha_near_r}
   \end{align}\\
\end{enumerate}
%
%
\end{definition}
Notice that $\theta_m \equiv \vartheta$ leads to $g_\Theta(t) = -\vartheta \log\left( 1-t \right) \in \mathcal{F}(1,\vartheta,0)$ 
and thus the Ewens measure is covered by the family $\mathcal{F}(r,\vartheta,K)$.
Also functions of the form $-\vartheta \log\left( 1-t \right) + f(t)$ with $f(t)$ holomorphic for $|t|<1+\epsilon$ are contained in $\mathcal{F}(1,\vartheta,f(1))$.
In particular, the case $\theta_k \neq\vartheta$ for only finitely many $k$ is covered by the family $\mathcal{F}(1,\vartheta,\cdot)$.
\begin{remark}
From from \cite[Theorem~VI.4]{FlSe09} we deduce the following observation, which we will use frequently:
if $g(t)$ is defined as in \eqref{eq:def_g_Theta_2} and $g(t)\in\mathcal{F}(r,\vartheta,K)$, then there exists some $\eps_m$ such that 
\begin{align}
 \label{eq:theta_k_to_vartheta}
\theta_m r^m = \vartheta +\eps_m \quad \text{ with } \ \eps_m \to 0 \ \text{ and } \ \sum_{m=1}^\infty \frac{|\eps_m|}{m} < \infty.
\end{align}
\end{remark}
We are now ready to state the main theorem of this section.
\begin{theorem}
\label{thm:total_cycles_asymp_with_restriction}
Let $g(t)$ in $\mathcal{F}(r,\vartheta, K)$ and $(D_n^{(j)})_{n\in\N, 1 \leq j \leq k}$ with $D_n^{(j)} \subset \set{1,\dots,n}$ be given. We define 
\begin{align}
\label{eq:def_d_n}
d_n^{(j)}:= 
\begin{cases}
 \max D_n^{(j)}  & \text{if }D_n^{(j)}\neq \emptyset,\\
 1 & \text{ otherwise },
\end{cases}
\quad \text{ and } \quad 
\bar{d}_n := \max \{d_n^{(j)}\}.  
\end{align}
Let further
\begin{align}\label{eq:replace}
G_n(t,w,v_1,...,v_k) := \exp \left( w g(t) + \sum_{j=1}^k{v_j L_{D_n^{(j)}}(t)}\right) 
\end{align}
with $w, v_1,\dots,v_k\in\C$ and $L_{D_n^{(j)}}(t)$ as in Lemma~\ref{lem:generating_hn_A}. Suppose that for each $C\in\R$
\begin{align}
\label{eq:assumption_on_dn}
C \log n- \frac{n}{d_n} \to -\infty 
\end{align}
%
holds as $n \to \infty$. Then we have for any fixed $b\in\N$
\begin{align}
\label{eq_asymp_Gn}
[t^{n-b}]\left[ G_n(t,w,v_1,...,v_k) \right]& \\
= \frac{e^{Kw } n^{w \vartheta -1}}{r^{n-b}}   
&\exp\Bigg( \sum_{j=1}^k{v_j L_{D_n^{(j)}}(r)} \Bigg)  
\Bigg(\frac{1}{\Gamma(w \vartheta)} +  O\bigg(\frac{\bar{d}_n}{n} \bigg) \Bigg) \nonumber
\end{align}
uniformly for bounded $|w|, |v_1| ,..., |v_k| \leq \hat{r}$ for some $\hat{r} > 0$.\\
\end{theorem}

\begin{remark}
 We have introduced in Theorem~\ref{thm:total_cycles_asymp_with_restriction} more than one set $D_n$ since this will allow us to compute easily
the finite dimensional distributions of the process $B_n(x)$ in Section~\ref{sec:func_limit_theorem}.
\end{remark}

\begin{remark}
Apparently assumption \eqref{eq:assumption_on_dn} is not fulfilled if $D_n^{(j)} = \set{1,\dots,n}$ for some $j$. However, in this case
\begin{align*}
 \nth{\exp\left( vL_{D_n^{(j)}}(t)\right)} =  \nth{\exp\left( v g(t) \right)} 
\end{align*}
holds and we can thus handle $D_n^{(j)} = \set{1,\dots,n}$ with Theorem~\ref{thm:total_cycles_asymp_with_restriction} by replacing $L_{D_n^{(j)}}(t)$ with $g(t)$ in \eqref{eq:replace}.
\end{remark}

\begin{remark}
 We will mostly use \eqref{eq_asymp_Gn} with $b = 0$, except in Subsection~\ref{sec:large_cycles} where we study the behaviour of the large cycles.
\end{remark}

%


Before proving Theorem~\ref{thm:asymp_cycle_counts} we deduce the large-$n$ behaviour of $h_n$.
\begin{corollary}
\label{cor:haviour_of_hn}
Let $g_\Theta(t)$ in $\mathcal{F}(r,\vartheta, K)$ be given and let $(A_n)_{n\in\N}$ be the defining sets of the measures $\PTA{A_n}{.}$ (see Definition~\ref{def_A-weighted_probabililty_measure}). We set $D_n:=\set{1,\dots,n}\setminus A_n$ and define $d_n$ as in \eqref{eq:def_Dn}.
If the sequence $d_n$ fulfils the assumption \eqref{eq:assumption_on_dn}, then
\begin{align*}
 h_n (A_n)
=
\exp\left( - L_{D_n}(r) \right) \frac{n^{\vartheta -1} e^{K } }{r^n \Gamma(\vartheta)} \left(1 +  O\left(\frac{d_n}{n} \right) \right).
\end{align*}
In particular, \eqref{eq:assumption_on_dn} is fulfilled if
\begin{align*}
 d_n \sim \log n \quad \text{ or } \quad d_n \sim n^{\alpha} \text{ for } 0 < \alpha <1.
\end{align*}
\end{corollary}

\begin{proof}
We know from Lemma~\ref{lem:generating_hn_A} that
\begin{align*}
 h_n (A) 
 =
\nth{\exp \left(g_\Theta(t) - L_{D}(t)\right)}
\end{align*}
holds for arbitrary sets $A \subset \N$. Thus the first part of the corollary follows immediately from Theorem~\ref{thm:total_cycles_asymp_with_restriction}.
The second part is obvious.
\end{proof}

\begin{proof}[Proof of Theorem~\ref{thm:total_cycles_asymp_with_restriction}]
For simplicity, we assume $k=1, b=0$ and write $D_n:= D_n^{(1)},d_n := \bar{d}_n = d^{(1)}_n$ and $v:=v_1$. The proof of the general case is completely similar. We apply Cauchy's integral formula to $G_n(t,w,v)$. This gives 
\begin{align}
\label{eq_Expect_with_cauchy}
\nth{ G_n(t,w,v) }
=
\frac{1}{2\pi i} \int_{\gamma} \exp\left( w g(t) +v  L_{D_n}(t) \right) \, \frac{dt}{t^{n+1}}
\end{align}
for some curve $\gamma$. 
We follow the idea in \cite[Section~VI.3]{FlSe09} and choose the curve $\gamma$ as in Figure~\ref{fig:curve_flajolet_4}.
%
%
\begin{figure}[h]
\centering
\subfigure[\,$\gamma=\gamma_1\cup\gamma_2\cup\gamma_3\cup\gamma_4$]{\hspace{-1pc}
\includegraphics[height=.18\textheight]{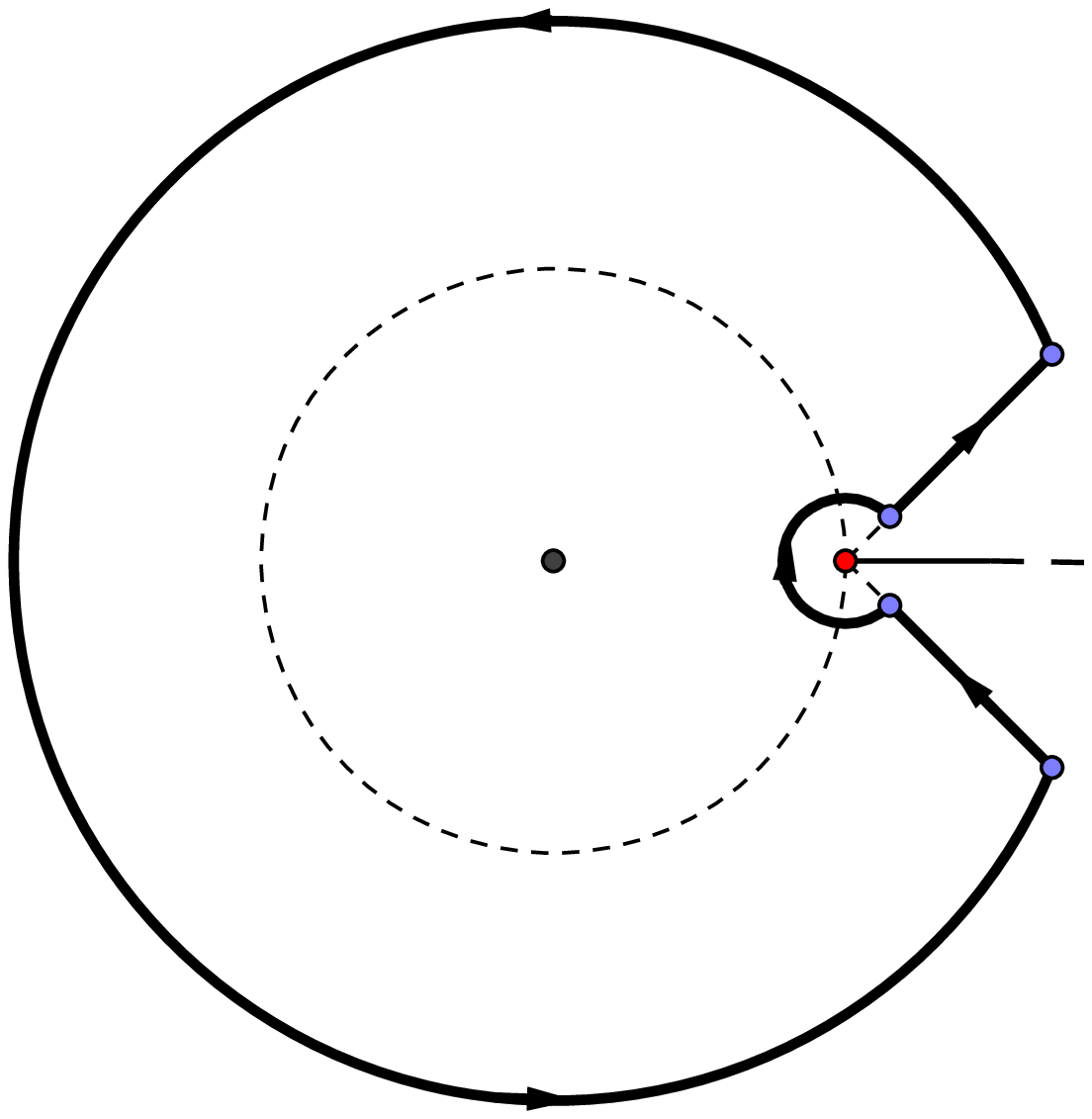}
 \label{fig:curve_flajolet_4}
 \put(-75,57){\mbox{\scriptsize$0$}}
  \put(-70,80){\mbox{\scriptsize$|z|= r$}}
  \put(-106,103){\mbox{\scriptsize$|z|= R^{\prime}$}}
  \put(-47,51){\mbox{\scriptsize$\gamma_3$}}
   \put(-25,69){\mbox{\scriptsize$\gamma_4$}}
   \put(-25,35){\mbox{\scriptsize$\gamma_2$}}
   \put(-100,51){\mbox{\scriptsize$\gamma_1$}}
 }
\hspace{2.5pc}\subfigure[\,$\gamma'=\gamma'_1\cup\gamma'_2\cup\gamma'_3$]{
   \includegraphics[height=.18\textheight]{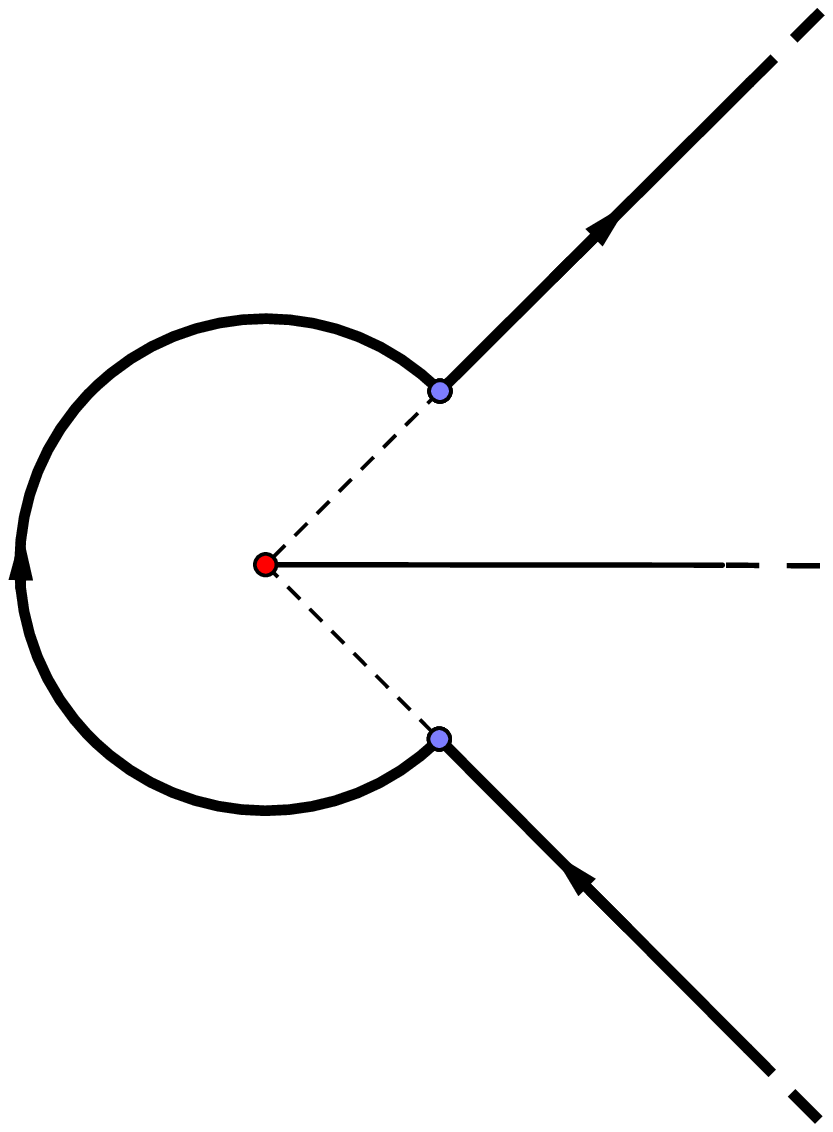}
   \label{fig:gamma'}
   \put(-58,42){\mbox{\scriptsize$0$}}
   \put(-86,80){\mbox{\scriptsize$|w| =1$}}
    \put(-26,27){\mbox{\scriptsize$\gamma'_1$}}
   \put(-91,50){\mbox{\scriptsize$\gamma'_2$}}
   \put(-26,95){\mbox{\scriptsize$\gamma'_3$}}
 }
\hspace{2.5pc}\subfigure[\,$\gamma''=\gamma''_1\cup\gamma''_2\cup\gamma''_3$]{
   \includegraphics[height=.18\textheight]{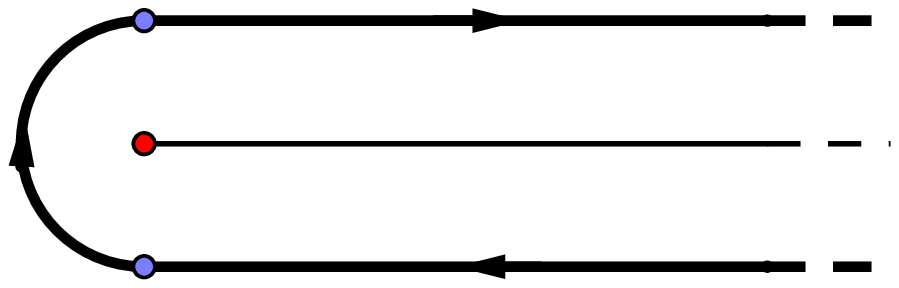}
   \label{fig:gamma''}
    \put(-79,49.5){\mbox{\scriptsize$0$}}
    \put(-46,30){\mbox{\scriptsize$\gamma''_1$}}
    \put(-98,49.5){\mbox{\scriptsize$\gamma''_2$}}
    \put(-46,71.0){\mbox{\scriptsize$\gamma''_3$}}
 }
\caption{{The curves used in the proof of Theorem~\ref{thm:total_cycles_asymp_with_restriction}}.}
\end{figure}
The main difference with \cite{FlSe09} is that we let the radius of the large circle slowly tend to $r$ while it is fixed in \cite{FlSe09}.
More precisely, by assumption $g(t)$ is holomorphic in $\Delta_0(r,R,\phi)$ (see \eqref{eq_def_delta_0}) and continuous on $\overline{\Delta_0(r,R,\phi)} \setminus\set{r}$.
We then define the radius of the large circle as
\begin{align*}
   R(n) := \min\set{r (1+ 1/d_n), R}
\end{align*}
and define further $\gamma$ as
\begin{align*}
 \gamma_1(\varphi) 
&:=
 R(n) e^{i\varphi} \ &\text{ for }& \ \varphi\in[-\pi+\alpha_n, \pi -\alpha_n]\\
\gamma_2(\varphi) 
&:=
 r\left(1 -\frac{1}{n} e^{-i\varphi} \right) \ &\text{ for }& \ \varphi\in[-\pi+\phi, \pi -\phi]\\
 \gamma_3(x) 
&:=
 r(1 + xe^{i\phi}) \ &\text{ for }& \ x \in[1/n, \hat{r}_n]\\
\gamma_4(x) 
&:=
r(1 + \bigl(\hat{r}_n-x \bigr)e^{-i\phi}) \ &\text{ for }& \ x \in[0,\hat{r}_n -1/n]
\end{align*}
where $\alpha_n$ and $\hat{r}_n$ are chosen such that the curve $\gamma$ is closed, i.e. $r + \hat{r}_ne^{i\phi} = R(n) e^{i\alpha_n}$.

We first compute the integral over $\gamma_1$. 
If $\sup_n d_n =C < \infty$, we clearly have $R(n) \geq \widetilde{R} > r$ for some $\widetilde{R}$ independent of $n$. 
Thus all points of the curve $\gamma_1$ have at least a distance $|\widetilde{R} - r|>0$ from $r$.
Therefore $g(t)$ is uniformly bounded on $\gamma_1$.
Furthermore $L_{D_n}(t) $ involves only $\theta_m$ with $m\leq C$ and is thus also uniformly bounded.
We get
\begin{align*}
\left|\frac{1}{2\pi i} \int_{\gamma_1} \exp\left( w g(t) +v  L_{D_n}(t) \right) \, \frac{dt}{t^{n+1}} \right| 
&= 
O(\widetilde{R}^{-n}) 
= 
O \left( \frac{n^{w \vartheta -2}}{r^n} e^{ v L_{D_n}(r) } \right).
\end{align*}
If $\sup_n d_n = \infty$ we have to be more careful. In this case
\begin{align*}
&\left|\frac{1}{2\pi i} \int_{\gamma_1} \exp\left( w g(t) +v  L_{D_n}(t) \right) \, \frac{dt}{t^{n+1}} \right| \\
&\leq 
\frac{1}{2\pi \bigl(R(n)\bigr)^n} 
\int_{-\pi+\alpha_n}^{\pi-\alpha_n} 
\left|\exp\left( w g\bigl(R(n)e^{i\varphi}\bigr) +v  L_{D_n}\bigl(R(n)e^{i\varphi}\bigr) \right) \right|  \, d\varphi\\
&\leq 
\frac{1}{\bigl(R(n)\bigr)^n} \exp(\max_{-\pi+\alpha_n \leq \varphi\leq \pi-\alpha_n } \{\Re\left[w g\bigl(R(n)e^{i\varphi}\bigr) +v  L_{D_n}\bigl(R(n)e^{i\varphi}\bigr)\right]\}).
\end{align*}
Using that $g(t)\in\mathcal{F}(r,\vartheta,K)$ is continuous on $\overline{\Delta_0(r,R,\phi)} \setminus\set{r}$ and 
the expansion of $g(t)$ around $r$, one immediately obtains
\begin{align*}
 \Re(g(t)) \leq \vartheta \log\left| \frac{1}{1-t/r} \right| + O(1) \ \text{ and } \
\Im(g(t)) =O(1)
\quad \text{ for all }t\in\Delta_0 \setminus\set{r}.
\end{align*}
This yields
\begin{align*}
 \Re[w g\bigl(R(n)e^{i\varphi}\bigr)] 
\leq 
|\Re(w)| \vartheta \log\left| \frac{1}{1-R(n)/r} \right| + O(1)
=
|\Re(w)| \vartheta \log d_n + O(1).
\end{align*}
Furthermore, we get
\begin{align}
\label{eq_integral_gamma_1_a}
\left | L_{D_n} \left( \left(1+ 1/d_n \right) r e^{i\varphi} \right) \right|
&\leq
\sum_{m\in D_n} \frac{\theta_m}{m} r^m \left( 1+ 1/d_n \right)^m
\nonumber\\
&\leq
\sum_{m\in D_n} \frac{\theta_m}{m} r^m \left(1+ O\left( m/d_n \right)  \right)
\nonumber\\
&\leq
L_{D_n}(r) + O(1)
\leq
\log d_n + O(1)
\end{align}
since $m \leq d_n$ and $\theta_m r^m\sim\vartheta$, see \eqref{eq:theta_k_to_vartheta}. We also have
\begin{align*}
(R(n)\bigr)^{-n}
\leq 
r^{-n}(1 + 1/d_n)^{-n} = r^{-n}\exp(-n \log(1 + 1/d_n)) \leq  r^{-n} \exp\left(-\frac{n}{2d_n} \right).
\end{align*}
Combining the above computations, we obtain
\begin{align*}
&\left|\frac{1}{2\pi i} \int_{\gamma_1} e^{ w g(t) +v  L_{D_n}(t)} \, \frac{dt}{t^{n+1}} \right| 
&=
O\left(r^{-n} \exp\left(-\frac{n}{2d_n} +(|\Re(w)|+|v|) \vartheta \log d_n \right)\right).
\end{align*}
It remains to prove that this is $O(n^{w\vartheta-2} \exp\bigl(-vL_{D_n}(r)\bigr)$. This holds if
\begin{align*}
 -\frac{n}{2d_n} +(|\Re(w)|\vartheta+|v|)  \log d_n \leq \Re((w\vartheta-2)\log n -vL_{D_n}(r)) + O(1)
\end{align*}
but this follows immediately from assumption \eqref{eq:assumption_on_dn} 
since $$L_{D_n}(r) \leq \log d_n \leq \log n.$$

The computations of the integrals over $\gamma_2,\gamma_3$ and $\gamma_4$ are completely similar to the computations in the proof of Theorem~VI.3 in \cite{FlSe09} and we thus give only a short overview. A simple calculation gives
\begin{align}
\label{eq_integral_gamma_2_a}
L_{D_n} \bigl(\gamma_2(\varphi)\bigr)
= 
L_{D_n}(r)
+ O\left(d_n/n \right)
\ \text{ and } \
 L_{D_n} \bigl( \gamma_3(x)\bigr)
=
 L_{D_n}(r) +    O\left(d_n x \right).
\end{align}
This observations together with the computations in \cite{FlSe09} then yields
\begin{align*}
\frac{1}{2\pi i} \int_{\gamma_2\cup \gamma_3\cup\gamma_4} G_n(t,w,v) \ \frac{dt}{t^{n+1}}
=
\frac{1}{2\pi i} \frac{n^{w\vartheta-1}}{r^n} \int_{\gamma'}  z^{-w\vartheta} e^{ z } dt \ \left(1+O\left( d_n/n \right) \right)
\end{align*}
with $\gamma'$ as in Figure~\ref{fig:gamma'}. 
The variable substitution $x= -z$ and a simple contour argument then gives
\begin{align*}
 \frac{1}{2\pi i} \frac{n^{w\vartheta-1}}{r^n} \int_{\gamma'}  z^{-w\vartheta} e^{ z } dt
= \frac{1}{2\pi i } \int_{\gamma''}  (-x)^{-w\vartheta} e^{-x} dt 
= \frac{1}{\Gamma(w\vartheta)}
\end{align*}
with $\gamma''$ as in Figure~\ref{fig:gamma''}. 
We have used in the second equality that that the integral is a well know expression for the inverse of $\Gamma$-function. 
Further details can be found for instance in \cite[Section~B.3]{FlSe09}.
\end{proof}

To investigate the bahaviour of the large cycles and a functional central limit theorem we have to consider in Section~\ref{sec:large_cycles} and \ref{sec:funct_limit_without_restrictions} expressions of the form
\begin{align}
 \nth{f(t) \cdot\exp \left( g(t) + v L_{D_n}(t)\right) }
\end{align}
where the funciton $f(t)$ is either a polynomial depending on $n$ or it is independent of $n$ and behaves like a derivative of the logarithm near $r$.
By suitable modifications of Theorem~\ref{thm:total_cycles_asymp_with_restriction} we obtain in this case the following asymptotics.

\begin{corollary}
\label{prop:tn_of_addintional_f}
Let the assumptions of Theorem~\ref{thm:total_cycles_asymp_with_restriction} be fulfilled with $k=1$ and 
write $D_n:= D_n^{(1)},d_n := \bar{d}_n = d^{(1)}_n$. If $f(t)$ is a holomorphic function in $\Delta_0$ and there exists a constant $\beta\geq 0$ such that
  \begin{align}\label{eq:f_near_r}
    f(t) = (1-t/r)^{-\beta}(1+ O(t-r)), \quad t\to r \ \text{ and } \ t\in\Delta_0,
  \end{align}
then
\begin{align}
&\nth{f(t) \cdot\exp \bigl( g(t) +v L_{D_n}(t)\bigr) }\nonumber\\
&= 
\frac{e^{K} n^{\vartheta + \beta -1}}{r^n}   
\exp\big( v L_{D_n}(r)\big)  \Bigg(\frac{1}{\Gamma(\vartheta+\beta)} +  O\bigg(\frac{d_n}{n} \bigg) \Bigg) .
\end{align}
\end{corollary}

\begin{proof}
Since we have $g(t) \in\mathcal{F}(r,\vartheta,K)$, we get with \eqref{eq_class_F_r_alpha_near_r} and \eqref{eq:f_near_r}
\begin{align*}
\log f(t)  + g(t)
=
-(\vartheta + \beta) \log(1-t/r) + K + O \left( t-r \right)   \text{ as } t\to r.
\end{align*}
We thus see that $\log f(t)  + g(t) \in\mathcal{F}\bigl(r,\vartheta + \beta,K\bigr)$ and the corollary follows immediately from
Theorem~\ref{thm:total_cycles_asymp_with_restriction} with $g(t)$ replaced by $\log f(t)  + g(t)$.
\end{proof}

\begin{corollary}
\label{prop:tn_of_addintional_poly}
Let the assumptions of Theorem~\ref{thm:total_cycles_asymp_with_restriction} be fulfilled with $k=1$ and 
write $D_n:= D_n^{(1)},d_n := \bar{d}_n = d^{(1)}_n$. Let further $P_n(t)$ be a sequence of polynomials with
$$
P_n(t)= \sum_{k} p_{k,n} t^k, \quad \text{ where }\  p_{n,k} \geq 0,
$$ 
such that $P_n\bigl(r(1+1/d_n)\bigr) = P_n(r) (1+o(1))$. We then have for each $v\in\R$
\begin{align}
&\nth{P_n(t) \cdot\exp \left( g(t) + v L_{D_n}(t)\right) }\nonumber\\
= 
&P_n(r)  \frac{n^{\vartheta-1}}{r^n} \exp\big(v L_{D_n}(r) \bigr) 
\Bigg(\frac{1}{\Gamma(w \vartheta)} + o(1) \Bigg)  
\end{align}
\end{corollary}

\begin{proof}
Since the computations for this proof are very similar to those of the proof of Theorem~\ref{thm:total_cycles_asymp_with_restriction}, 
we only illustrate the estimate over $\gamma_1$ with $\gamma_1(\varphi) = r(1+ 1/d_n)e^{i\varphi}$. 
We argue as in the first part of the proof of Theorem~\ref{thm:total_cycles_asymp_with_restriction} and use that $P_n(t)$ has positive coefficients. We obtain
%
%
\begin{align*}
&\left|\frac{1}{2\pi i} \int_{\gamma_1} P_n(t) \exp\big(g(t) + v L_{D_n}(t)\big) \, \frac{dt}{t^n}\right| \nonumber\\
&\leq
\frac{P_n\big(r (1+ 1/d_n)\bigr)}{2\pi \bigl(R(n)\bigr)^n} 
\int_{-\pi+\alpha_n}^{\pi-\alpha_n} 
\left|\exp\left( w g\bigl(R(n)e^{i\varphi}\bigr) +v  L_{D_n}\bigl(R(n)e^{i\varphi}\bigr) \right) \right|  \, d\varphi
\end{align*}
The latter integral is now the same as in the proof of Theorem~\ref{thm:total_cycles_asymp_with_restriction}.
Using the estimate in the proof of Theorem~\ref{thm:total_cycles_asymp_with_restriction} and the assumption on $P_n(t)$ then completes the proof. 
\end{proof}

\section{Cycle counts, total number of cycles and large cycles}
\label{sec_cycles}

%
%
%
%

\subsection{The cycle counts and the total number of cycles}
\label{sec_cycle_counts}

We consider here the cycle counts $C_m$ and the total number of cycles $T$, defined in \eqref{eq_def_Cm_lambda}. 
First, we compute their generating functions and then deduce with Theorem~\ref{thm:total_cycles_asymp_with_restriction} their asymptotic behaviour. 
As mentioned in the introduction, the required computations are quite similar to those in \cite{NiZe11}. 
Therefore, we give here only a short overview and refer to \cite{NiZe11} for more details.

\begin{lemma}
\label{lem:gen_cycle_counts}
Let $A\subset \N$, $D := \N \setminus A$ be given. We then have for $w\in\C$ as formal power series 
\begin{align*}
\sum_{n=1}^\infty h_n(A) \ETA{A}{\exp(w T)} t^n
=
\exp\left(e^w g_\Theta(t) - e^w L_{D}(t) \right).
\end{align*}
Let further $M = \set{m_1,\dots,m_d} \subset A$ be given.
We then have for $w_1,\dots,w_d \in\C$ as formal power series
\begin{align*}
 \sum_{n=1}^\infty h_n(A) & \ETA{A}{\exp \left(\sum_{j=1}^d w_{m_j} C_{m_j} \right)} t^n \\
= &
\exp\left( \sum_{j=1}^d \frac{\theta_{m_j}}{m_j} (e^{w_{m_j}}-1)t^{m_j}   \right)
\exp\left(g_\Theta(t) - L_{D}(t) \right).
\end{align*}
\end{lemma}
We omit the proof since it is a simple application of Lemma~\ref{lem:cycle_index_theorem} and 
the computations are similar to the proof of Lemma~\ref{lem:generating_hn_A}.

It follows with Lemma~\ref{lem:gen_cycle_counts} that 
\begin{align}
\label{eq:what_is_a_good_name}
h_n(A) \ETA{A}{\exp(w T)}
=
[t^n]\left[
\exp\left(e^w g_\Theta(t) - e^w L_{D}(t) \right)
\right]
\end{align}
with $A\subset \N$ arbitrary. We can thus  replace $A$ in \eqref{eq:what_is_a_good_name} by any  $A_n$ depending on $n$.
Note that this is not possible in Lemma~\ref{lem:gen_cycle_counts}. Now combine 
Theorem~\ref{thm:total_cycles_asymp_with_restriction} and Lemma~\ref{lem:gen_cycle_counts} to obtain the asymptotic behaviour of the cycle counts.

\begin{theorem}
\label{thm:asymp_cycle_counts}
Suppose that $g_\Theta(t)$ is in $\mathcal{F}(r,\vartheta, K)$. 
Let further $M = \set{m_1,\dots,m_d}$ and $(A_n)_{n\in\N}$ with $A_n \subset \set{1,...,n}$ be given and let $d_n$ be defined as in \eqref{eq:def_Dn}. Suppose that
\begin{enumerate}
 \item $d_n$ fulfils the assumption \eqref{eq:assumption_on_dn} and
 \item there exists $n_0\in\N$ such that $M\subset A_n$ for all $n \geq n_0$.
\end{enumerate}
Then
\begin{align}
\label{eq_asympt_Cm}
\ETA{A_n}{\exp\left(\sum_{j=1}^d w_{m_j} C_{m_j} \right)}
=
\exp \left( \sum_{j=1}^d \frac{\theta_{m_j}}{m_j} (e^{w_{m_j}}-1)r^{m_j}   \right) + O\left( \frac{d_n}{n} \right)
\end{align}
uniformly in $w_{m_1},\dots,w_{m_d}$ for bounded $\Re(w_{m_1}),\cdots,\Re(w_{m_d})$.
In particular, the random variables $C_{m_j}, m_j\in M$ converge in law to independent Poisson distributed random variables $Y_{m_j}$ with
$\E{Y_{m_j}} = \theta_{m_j} r^{m_j} / m_j$.
\end{theorem}

\begin{proof}
Equation \eqref{eq_asympt_Cm} follows immediately from Lemma~\ref{lem:gen_cycle_counts} and Theorem~\ref{thm:total_cycles_asymp_with_restriction}.
The error is uniform for bounded $\Re(w_{m_1}),\cdots,\Re(w_{m_d})$ since all $C_{m_j}\in\N$ and thus the function on the left-hand side of \eqref{eq_asympt_Cm} is periodic.
\end{proof}

The asymptotic behaviour of the total cycle number $T$ is computed analogously.

\begin{theorem}\label{thm_charfunctionK0n}
Let $g_\Theta(t)$, $(A_n)_{n\in\N}$ and $d_n$ be defined as in Theorem~\ref{thm:asymp_cycle_counts}. Then
\begin{align}\label{eq_charfunctionK0n}
 \ETA{A_n}{\exp(is T)}
=
n^{\vartheta(e^{is} -1)} e^{(K-L_{D_n}(r))(e^{is}-1)} \left( \frac{\Gamma(\vartheta)}{\Gamma(e^{is}\vartheta)}  + O\left( \frac{d_n}{n} \right)  \right)
\end{align}
uniformly in $s$ for bounded $\Re(is)$.
\end{theorem}

We will not give the proof here since it is quite similar to that of the previous theorem. However, an analogue result, Theorem~\ref{thm_charfunctionBx}, is proved in Section~\ref{sec:func_limit_theorem}.

Given the characteristic function of the total cycle number, one can show the following central limit theorem, in analogy to Theorem 4.2 in \cite{NiZe11}.

\begin{corollary}\label{cor:cltKn}
 Under the same assumptions as in Theorem~\ref{thm:asymp_cycle_counts}, we have
\begin{align*}
  \frac{T - \vartheta \log n}{\sqrt{\vartheta \log n }} \overset{d}{\longrightarrow} \mathcal{N}(0,1),
\end{align*}
where $\overset{d}{\longrightarrow}$ denotes convergence in distribution and $\mathcal{N}(0,1)$ a standard normal random variable.
\end{corollary}

We will state and prove in Section~\ref{sec:func_limit_theorem} a similar result, see Corollary~\ref{cor_cltBn}.
We thus omit the proof here.
In fact, still in analogy to \cite{NiZe11}, it follows immediately from equation \eqref{eq_charfunctionK0n} that $T$ converges in a stronger sense, namely it is mod-Poisson convergent.

\begin{corollary}\label{cor:modT}
 Under the same assumptions as in Theorem~\ref{thm:asymp_cycle_counts}, the sequence $(T_n)_{n \in \N}$ converges in the strong mod-Poisson sense with parameter $K + \vartheta \log n - L_{D_n}(r)$ and limiting function $\Gamma(\vartheta) / \Gamma(\vartheta e^{is})$.
\end{corollary}
%

As in \cite{NiZe11} one can now approximate $T_n$ by a Poisson random variable with mean $K + \vartheta \log n - L_{D_n}(t)$ or compute large deviations estimates. But as all  computations are completely similar and we thus omit them. However, in Section~\ref{sec:func_limit_theorem} we state and prove an analogue result, see Corollary~\ref{cor_modBx}.

\subsection{Behaviour of large cycles}
\label{sec:large_cycles}

The goal of this section is to study the asymptotic behaviour of the large
cycles. The main result, Theorem~\ref{thm:long_cycles}, yields 
the same asymptotic behaviour as in the Ewens case, see for instance Vershik and
Shmidt~\cite{ShVe77} and Kingman~\cite{Ki77}.

Let $\ell^{(1)}(\sigma)$ be the length of the longest cycle of $\sigma\in \Sn$, $\ell^{(2)}(\sigma)$ the length of the second longest cycle and so on.
If $\sigma$ has cycle type $\la = (\la_1,\la_2,\cdots)$ this means that $\ell^{(j)} = \la_j$ for $j\in\N$.

\begin{theorem}
\label{thm:long_cycles}
Let $(A_n)_{n\in\N}$ be the defining sets of the measures $\mathbb{P}_{\Theta}^{(A_n)}$ and define
$D_n:=\set{1,\dots,n}\setminus A_n$ and $d_n$ as in \eqref{eq:def_Dn}. Suppose that 
$g_\Theta(t)\in\mathcal{F}(r,\vartheta, K)$, $d_n$ fulfills assumption~\ref{eq:assumption_on_dn} and that for all $b \geq 0$
\begin{align}
\label{eq:assume_g_derivative}
 \left(\frac{\partial}{\partial t} \right)^{b+1} g_\Theta(t) 
= 
 \frac{ \vartheta b!}{r^{b+1}(1-t/r)^{b+1}} (1+ O(t-r))
\end{align}
as $t \to r$. We then have, as $n \rightarrow \infty$,
\begin{align*}
 \left( \frac{\ell^{(1)}}{n}, \frac{\ell^{(2)}}{n},\dots \right) 
\stackrel{d}{\longrightarrow}
\mathcal{PD}(\vartheta)
\end{align*}
where $\mathcal{PD}(\vartheta)$ denotes the Poisson-Dirichlet distribution with parameter $\vartheta$ (see \cite{Bi99}).
\end{theorem}

\begin{proof}
Let $\ell_1 = \ell_1(\sigma)$ be the length of the cycle containing $1$,  $\ell_2 = \ell_2(\sigma)$ containing the least element not contained in the cycle containing $1$ and so on.
We prove that for each fixed $m \in \N$, as $n \rightarrow \infty$,
\begin{align}
\label{eq:convergence_tp_beta}
 \left( \frac{\ell_1}{n}, \frac{\ell_2}{n-\ell_1},\dots ,  \frac{\ell_m}{n-\sum_{j=1}^{m-1}\ell_j} \right) 
\stackrel{d}{\longrightarrow}
(B_1,\dots,B_m)
\end{align}
holds, where $B_1,\dots,B_m$ are independent beta random variables with parameter $(1,\vartheta)$.
This result immediately implies the theorem, see for instance \cite{ShVe77}.

We start with the case $m=1$. We first compute the distribution of $\ell_1$.
If $k\in A_n$ is given, then there are $(n-1)\cdots(n-k+1)$ possible cycles of
length $k$ containing the element~$1$, and the choice of such a
cycle does not influence the cycle lengths of the remaining cycles.
Using the definition of $h_n(A_n)$ and a small computation then gives
\begin{align*}
 \PTA{A_n}{\ell_1 = k } =\frac{\theta_k}{n} \frac{h_{n-k}(A_n)}{h_{n}(A_n)} \one_{\set{k\in A_n}}.
\end{align*}
We use the Pochhammer symbol $(k)_{b} = k (k-1)\cdots(k-b+1)$ and get for $b\geq1$
\begin{align*}
\ETA{A_n}{(\ell_1-1)_b} = \frac{1}{n}\sum_{k=b+1}^n (k-1)_b \cdot \theta_k \one_{\set{k\in A_n}} \frac{h_{n-k}(A_n)}{ h_{n}(A_n) }.
\end{align*}
On the other hand we have 
\begin{align*}
 \left(\frac{\partial}{\partial t} \right)^{b+1} g_\Theta(t)   
=
\sum_{k=b+1}^\infty (k-1)_b \cdot \theta_k t^{k-b-1}.
\end{align*}
This together with the definition of $L_{D_n}(t)$ and Lemma~\ref{lem:generating_hn_A} gives
\begin{align*}
h_{n}(A_n) \ETA{A_n}{(\ell_1-1)_b} 
=
[t^{n-b-1}]\left[ \frac{e^{g_\Theta(t) - L_{D_n}(t)}}{n} \left(\frac{\partial}{\partial t} \right)^{b+1} \Big(g_\Theta(t) - L_{D_n}(t)\Big)   \right].
\end{align*}

We  can now use Corollary~\ref{prop:tn_of_addintional_f} and~\ref{prop:tn_of_addintional_poly} to compute the asymptotic behaviour of this expression.
If follows with Corollary~\ref{prop:tn_of_addintional_f} and assumption \eqref{eq:assume_g_derivative} that
\begin{align}
\label{eq:large_moments_combined}
 [t^{n-b-1}]\left[ \exp\left(g_\Theta(t) -  L_{D_n}(t)\right) \left (\frac{\partial}{\partial t} \right)^{b+1} g_\Theta(t)    \right]& \nonumber\\
= 
\vartheta b! \frac{n^{\vartheta + b } e^K}{r^{n} \Gamma(\vartheta +b +1)} \exp\left( - L_{D_n}(r) \right)& \left( 1+ O(d_n/n) \right).
\end{align}
 
We show as next that the remaining part can be neglected with respect to \eqref{eq:large_moments_combined}.
We get with \eqref{eq:theta_k_to_vartheta}
\begin{align*}
 L_{D_n}(r) &= \sum_{k=1}^{d_n} \frac{\theta_k}{k}r^k = O\left(\sum_{k=1}^{d_n} \frac{1}{k} \right) = O(\log d_n)
\end{align*}
and
\begin{align*}
\left(\frac{\partial}{\partial t} \right)^{b+1} L_{D_n}(r)
&=
\sum_{k=1}^{d_n} (k-1)_b\theta_k r^{k-b-1}
= O\left( \sum_{k=1}^{d_n} \binom{k-1}{b} \right) 
= O\left( \binom{d_n}{b+1} \right) \\
&= O\Big((d_n)^{b+1}\Big).
\end{align*}
Using Corollary~\ref{prop:tn_of_addintional_poly} together with this computations gives
%
%
\begin{align*}
&[t^{n-b-1}]\left[ e^{g_\Theta(t) - L_{D_n}(t)} \left(\frac{\partial}{\partial t} \right)^{b+1} L_{D_n}(t)\right]
=
O\left(\frac{n^{\vartheta-1} (d_n)^{b+1} \exp\left(- L_{D_n}(r)\right)}{r^n}\right)
\end{align*}

Comparing this to \eqref{eq:large_moments_combined}, we see that we can neglect it since $d_n = o(n)$.
Thus the leading term of $h_{n}(A_n) \ETA{A_n}{(\ell_1-1)_b}$ comes from \eqref{eq:large_moments_combined} and combined with the asymptotic behaviour of $h_n$, see Corollary~\ref{cor:haviour_of_hn}, we obtain
\begin{align*}
 \ETA{A_n}{(\ell_1-1)_b} = \vartheta n^b \frac{b! \ \Gamma(\vartheta)}{\Gamma(\vartheta+b +1)} \left( 1+ O\left(\frac{d_n}{n}\right) \right). 
\end{align*}
It follows that
\begin{align*}
\ETA{A_n}{\left(\frac{\ell_1}{n}\right)^b} 
= 
\frac{b! \ \Gamma(\vartheta+1)}{\Gamma(\vartheta+b +1)} 
=
\E{B_1^b}
\end{align*}
with $B_1$ a beta random variable with parameter $\vartheta$. This completes the proof in the case $m=1$.

Equation \eqref{eq:convergence_tp_beta} now can be proved for arbitrary $m$ by induction over $m$.
The argumentation is (almost) the same as in the proof of Proposition~5.2 in \cite{BeUe10}. One only has to check that
\begin{align*}
&\mathbb{P}_{\Theta,n}^{(A_n)}\left[\frac{\ell_{m+1}}{n-\sum_{j=1}^{m} \ell j} \leq a_{m+1}  \Big| \ell = a_1,\dots,\ell_n = a_m\right]
\nonumber\\
&= 
\mathbb{P}_{\Theta,n-\sum_{j=1}^m a_j}^{(A_n)}\left[\frac{\ell_{m+1}}{n-\sum_{j=1}^{m} a_j} \leq a_{m+1} \right].
\end{align*}

\end{proof}

\section{A functional central limit theorem}
\label{sec:func_limit_theorem}

The object of this section is to prove that the number of cycles with length not exceeding $n^x$ converges, after normalisation, 
weakly to the standard Brownian motion with respect to the Skorohod topology.
(Details about the Skorohod topology and weak convergence of processes can be found for instance in \cite{Bi99}). 
Formally, this means we consider the functional
\begin{align}\label{eq:def_Bn}
 B_n(x) := \sum_{m=1}^{\lfloor n^x \rfloor}{C_m}.
\end{align} 
It was first shown by DeLaurentis and Pittel \cite{DePi85}, with respect to the uniform measure on $\Sn$, that the process
\begin{align}\label{eq:def_flt}
 \widetilde{B}_n(x) :=  \frac{B_n(x) - x \log n}{\sqrt{\log n}}
\end{align}
converges weakly to the standard Brownian motion for $0 \leq x \leq 1$. 
A corresponding result for the Ewens measure ($\theta_j = \vartheta$ for all $j \geq 1$) was shown by Hansen \cite{Ha90} and Donelly, Kurtz and Tavar\`{e} \cite{DoKuta91}. For this, $\log n $ in \eqref{eq:def_flt} needs to be replaced by $\vartheta \log n$. By an appropriate rescaling, we will show in this section the validity of an analogue result for our more general measure $\mathbb{P}_{\Theta}^{(A_n)}$ with the usual assumptions on the parameters $\theta_j$. 

\subsection{Without restriction}
\label{sec:funct_limit_without_restrictions}
Throughout this subsection we assume no restrictions on the cycle lengths, that is $A_n = \set{1,...,n}$ in Definition~\ref{def_A-weighted_probabililty_measure}, and write $\ET{.}$ instead of $\ETA{A_n}{.}$,  $\mathbb{P}_{\Theta}$ instead of $\mathbb{P}_{\Theta}^{(A_n)}$ and $h_n$ instead of $h_n(A_n)$. First, we compute the characteristic function of the process given by (\ref{eq:def_Bn}).

\begin{theorem}\label{thm_charfunctionBx}
 Suppose that $g_\Theta(t)$ is in $\mathcal{F}(r,\vartheta, K)$ and let $B_n(.)$ be defined as in \eqref{eq:def_Bn}. Then, for any fixed $x \in [0,1)$, we have
\begin{align}\label{eq_charfunctionBx}
  \ET{\exp(is B_n(x))} = \exp\left((e^{is}-1)L_{D_x}(r)\right)\left( 1 + O(n^{x-1})\right),
\end{align}
whith $D_x = \{1, ..., \lfloor n^x \rfloor\}$.\\
\end{theorem}
\begin{remark}
 Note that $B_n(1) = T$, and thus Theorem~\ref{thm_charfunctionK0n} states a similar behaviour as in \eqref{eq_charfunctionBx} for $x=1$, except that the $1$ on the right-hand side in \eqref{eq_charfunctionBx} is replaced by the quotient $\Gamma(\vartheta) / \Gamma(e^{is} \vartheta)$.
\end{remark}
\begin{proof}
 Consider $B_b := \sum_{k=1}^b{C_k}$. By Lemma~\ref{lem:size_of_conj_classes} we get
\begin{align*}
 h_n \ET{\exp(is B_b)} 
= \slan \exp \bigg(\sum_{k=1}^b{C_k}\bigg) \prod_{k = 1}^{\ell(\lambda)} \theta_{\lambda_k}.
\end{align*}
Now apply Lemma~\ref{lem:cycle_index_theorem} to obtain
\begin{align*}
 \sum_{n=0}^{\infty} {h_n \ET{\exp(is B_b)} t^n} 
&= \sla \bigg(\prod_{k=1}^b \big(\theta_k e^{is}\big)^{C_k}\bigg) \bigg( \prod_{k = b+1}^{\infty} \theta_k^{C_k} \bigg) \ t^{|\lambda|} \\
&= \exp\bigg(\sum_{k=1}^{b} \frac{\theta_k e^{is}}{k} t^k + \sum_{k=b+1}^{\infty} \frac{\theta_k }{k} t^k \bigg)\\
&= \exp\big((e^{is}-1)L_{D_b}(t) + g_{\Theta}(t)\big), 
\end{align*}
where $D_b = \{1, ..., b\}$. Then set $b = \lfloor n^x \rfloor$ and Theorem~\ref{thm:total_cycles_asymp_with_restriction} gives the result.
\end{proof}
\begin{corollary}\label{cor_modBx}
 Let $g_\Theta(t)$ and $B_n(.)$ be as in Theorem~\ref{thm_charfunctionBx}. Then, for any fixed $x \in [0,1)$, the sequence $(B_n(x))_{n \in \N}$ is strongly mod-Poisson convergent with limiting function $1$ and parameter $ L_{D_x}(r)$.
\end{corollary}
This corollary follows immediately from \eqref{eq_charfunctionBx}. Note again that a similar result for $x=1$ can be found in Corollary~\ref{cor:modT}. For the definition and details of mod-convergence we refer to \cite{BaKoNi}.

We obtain from \eqref{eq:theta_k_to_vartheta} that
\begin{align}\label{eq:L_x}
L_{D_x}(r) 
=  
\sum_{m = 1}^{\lfloor n^x \rfloor}{\frac{\vartheta}{m}} + \sum_{m = 1}^{\lfloor n^x \rfloor}{\frac{\epsilon_m}{m}}
= 
x \vartheta \log n + c +o(1)
\end{align}
as $n \to \infty$ with some $c\in\R$. This shows that the mod-Poisson convergence in Corollary~\ref{eq_charfunctionBx} does also hold with parameter $x \vartheta \log n + c$.
Given this, we can estimate the distance of $B_n(x)$ and a Poisson random variable with mean $x \vartheta \log n + c$, analogously to Lemma 4.6 in \cite{NiZe11}. 
This is done in terms of the \emph{point metric} $d_{loc}$ and the \emph{Kolmogorov distance} $d_K$.

\begin{corollary}
 Let $g_\Theta(t)$ and $B_n(.)$ be as in Theorem~\ref{thm_charfunctionBx} and let $P_{\vartheta}$ be a Poisson distributed random variable 
with mean $x \vartheta \log n + c$. Then, for any fixed $x \in [0,1)$, 
\begin{align*}
 d_{loc}\big(B_n(x),P_{\vartheta}\big) = O\left(\frac{n^{x-1}}{\log n}\right) \quad \text{ and } \quad
d_{K}\big(B_n(x),P_{\vartheta}\big) = O\left(\frac{n^{x-1}}{\log n}\right).
\end{align*}
\end{corollary}

\begin{proof}
 These estimates can be established with Proposition 3.1 and Corollary 3.2 in \cite{BaKoNi} with $\chi(s) = \exp((e^{is}-1)(x \vartheta \log n + \gamma + c))$, $\psi_{\nu}(s)=1$ and $\psi_{\mu}(s)=1$. 
\end{proof}

Another consequence of Theorem~\ref{thm_charfunctionBx} is the following central limit result.
\begin{corollary}\label{cor_cltBn}
 Let $g_\Theta(t)$ and $B_n(.)$ be as in Theorem~\ref{thm_charfunctionBx}. Then, for any fixed $x \in [0,1]$,
\begin{align*}
\widetilde{B}_n(x) :=  \frac{B_n(x) - x \vartheta \log n}{\sqrt{\vartheta \log n}} \overset{d}{\longrightarrow} \mathcal{N}(0,x), 
\end{align*}
where $\overset{d}{\longrightarrow}$ denotes convergence in distribution and $\mathcal{N}(0,x)$ a centred Gaussian random variable with variance $x$.
\end{corollary}

\begin{proof}
The case $x=1$ follows immediately from Corollary~\ref{cor:cltKn} and we  can thus assume $x<1$. We know from \eqref{eq:L_x} that
$L_{D_x}(r) =  x \vartheta \log n + O\left(1\right)$, as $n \to \infty$. We obtain with \eqref{eq_charfunctionBx} 
\begin{align*}
 \ET{\exp\left(\frac{is}{\sqrt{ \vartheta \log n}} B_n(x)\right)} 
&= \exp\left(\left(e^{\frac{is}{\sqrt{ \vartheta \log n}}}-1\right) x \vartheta \log  n \right)\left( 1 + o(1)\right)\\
&= \exp\left(is x \sqrt{\vartheta \log n} - \frac{s^2x}{2} \right)\left( 1 + o(1)\right)
\end{align*}
and the result follows.
\end{proof}

We now turn to the main result of this section. As already mentioned in the beginning of this section, it was shown that the process $\widetilde{B}_n(.)$ given in \eqref{eq:def_flt}, considered with respect to the uniform measure (and with respect to the Ewens measure, when $\widetilde{B}_n(.)$ is properly rescaled) converges weakly to the standard Brownian motion. In our setting, the analogue statement is the following.

\begin{theorem}
Suppose that $g_\Theta(t)$ is in $\mathcal{F}(r,\vartheta, K)$ and define
 \begin{align*}
\widetilde{B}_n(x) :=  \frac{B_n(x) - x \vartheta \log n}{\sqrt{\vartheta \log n}}.
\end{align*}
Then, as $n \rightarrow \infty$ and for $0 \leq x \leq 1$, $\widetilde{B}_n$ converges weakly to the standard Brownian motion $\mathcal{W}$ on $[0,1]$.
\end{theorem}

\begin{proof}
We will proof this statement following the arguments of Hansen \cite{Ha90}. We first define a process
\begin{align}\label{eq:def_Bn*}
 B_n^*(x):= \big(B_n(x) - L_{D_x}(r)\big)/\sqrt{\vartheta \log n}
\end{align}
with $D_x$ as in Theorem~\ref{thm_charfunctionBx}. It follows that
\begin{align*}
|\widetilde{B}_n(x) - B_n^*(x)| = |L_{D_x}(r) - x \vartheta \log n| / \sqrt{\vartheta \log n} = o(1), \ \text{ as }  n \to \infty
\end{align*}
with $o(1)$ uniform in $x\in[0,1]$. Therefore, the distance between $\widetilde{B}_n(x)$ and $B_n^*(x)$ is asymptotically vanishing with respect to the Skorohod topology on the space of right-continuous functions with left limits. It is thus sufficient to prove $B_n^* \overset{d}{\longrightarrow} \mathcal{W}$. We will proceed in two steps: first, we will show that the process $B_n^*(x)$ converges to $\mathcal{W}(.)$ in terms of finite-dimensional distributions and then its tightness.

%
%
%
%

\textit{Convergence of the finite dimensional distributions.}
We have to show that for any $k \in \N$ and $0 \leq x_1 < x_2 < ... < x_k \leq 1$ the random vector $\{B_n^*(x_j)\}_{j = 1}^k$ converges in distribution to the vector $\{\mathcal{N}(0,x_j)\}_{j = 1}^k$ with independent increments. We know from Corollary~\ref{cor_cltBn} that $B_n^*(x_j) \overset{d}{\longrightarrow} \mathcal{N}(0,x_j)$ for all $x_j \in [0,1]$.
It remains to show that the increments are independent. 
We define the sets $D_n^{(j)}:= \set{\lfloor n^{x_{j-1}} \rfloor +1,\dots,\lfloor n^{x_j} \rfloor }$ with $x_0:=0$ and a straight forward application of Lemma~\ref{lem:cycle_index_theorem} gives 
\begin{align}
\label{eq:char_increments}
h_n \ET{e^{\sum_{j=1}^{k} \bigl(i s_j( B_n(x_j)- B_n(x_{j-1}) \bigr)}}
=
\nth{e^{g_\Theta(t) + \sum_{j=1}^k (e^{is_j}-1) L_{D_n^{(j)}}(t)} }.
\end{align}
We  can thus apply Theorem~\ref{thm:total_cycles_asymp_with_restriction}. The remaining computations are the same as in the proof of Corollary~\ref{cor_cltBn}.
Only the case $x_k = 1$ needs further explanation. This is because  we have in this case $D_n^{(k)} = \set{n^{\lfloor x_{k-1}} \rfloor +1,\dots,n}$ and thus assumption \eqref{eq:assumption_on_dn} is not satisfied. However, we then have 
\begin{align*}
 L_{D_n^{(k)}}(t) = g_\Theta(t) - L_{D_c}(t) + t^{n+1} f_n(t)
\end{align*}
with $D_c:= \{1,\dots,\lfloor n^{x_{k-1}} \rfloor \}$ and $f_n(t)$ a holomorphic function around the origin. 
Inserting this into \eqref{eq:char_increments}, one can see that the term $t^{n+1} f_n(t)$ can be neglected. 
Hence, Theorem~\ref{thm:total_cycles_asymp_with_restriction} does also apply for $x_k=1$.
%
%
%
%

\textit{Tightness.}
It remains to prove that process $B_n^*(.)$ is tight. We use the moment condition given in \cite[Theorem~15.6]{Bi99}. More precisely, we show that for any $n \geq 0$ and $0 \leq x_1 < x < x_2 \leq 1$,
\begin{align}\label{eq_momentcondition}
 E_{\Theta}^{B_n^*} := \ET{\big(B_n^*(x) - B_n^*(x_1)\big)^2 \big(B_n^*(x_2) - B_n^*(x)\big)^2} = O \big( (x_2-x_1)^2 \big).
\end{align} 
We start with the identity
\begin{align}\label{eq:joint_second_moment}
h_n E_{\Theta}^{B_n^*} = \frac{[t^n]}{(\vartheta \log n)^2}\Big[ L_{D_1}(t) L_{D_2}(t) \ e^{g_{\Theta}(t)}\Big],
\end{align}
where $D_1 = \{\lfloor n^{x_1}\rfloor +1, ..., \lfloor n^x\rfloor \}$, $D_2 = \{\lfloor n^x \rfloor +1, ..., \lfloor n^{x_2}\rfloor \}$. Before we prove \eqref{eq:joint_second_moment}, we complete the proof of the tightness. It follows with 
Corollary~\ref{prop:tn_of_addintional_poly} that
\begin{align}
 [t^n]\Big[ L_{D_1}(t) L_{D_2}(t) \ e^{g_{\Theta}(t)}\Big] 
=  
O\bigg(L_{D_1}(r) L_{D_2}(r) \frac{n^{\vartheta-1}}{r^n}\bigg)
\label{eq:tight_upper_theta_neq_2}
\end{align}
and therefore
\begin{align}
\label{eq:joint_second_moment_upper_bound_2}
E_{\Theta}^{B_n^*} =  O\bigg(\frac{ L_{D_1}(r) L_{D_2}(r)}{(\log n)^2}\bigg).
\end{align}

Then, from equation \eqref{eq:theta_k_to_vartheta} follows
\begin{align}
 L_{D_1}(r) = O\big( (x-x_1)\log n \big) \quad \text{ and } \quad L_{D_2}(r) = O\big( (x_2-x)\log n\big).
\end{align}
Finally, this yields
\begin{align*}
E_{\Theta}^{B_n^*} = O\big( (x-x_1)(x_2-x) \big) = O\big( (x_2-x_1)^2 \big).
\end{align*}
which completes the proof of \eqref{eq_momentcondition} and proves the tightness.

It remains to prove \eqref{eq:joint_second_moment}. One can try to proceed with Lemma~\ref{lem:cycle_index_theorem}, but the computations are rather technical. We prefer to follow the idea of Hansen \cite{Ha90}. Therefore, 
we consider for $0<t<r$ a product space 
\begin{align*}
 \Omega_t := \set{(k_1, k_2, ...) | k_i \text{ is a non-negative integer}}
\end{align*}
and a measure $\mathbb{P}_{\Theta}^t$ on $\Omega_t$ given such that the $m-$th coordinate of $\Omega_t$ is Poisson distributed with parameter $\theta_m t^m / m$.  Then, analogously to \cite[Lemma~2.1]{Ha90}, we have
\begin{align}\label{eq_Pt}
\mathbb{P}_{\Theta}^t[v = n] = t^n h_n \exp(- g_{\Theta}(t)),
\end{align}
where $v:\Omega_t \to \N$  is defined by $v(k_1, k_2, ...) = \sum_{m=1}^{\infty}{mk_m}$. 
We omit the prove of \eqref{eq_Pt} since it is line by line the same as the proof of \cite[Lemma~2.1]{Ha90}, one simply has to replace $\theta$ by $\theta_k$ and $(1-t)^{-\theta}$ by $\exp(g_\Theta(t))$.

We obtain the following identity between $\mathbb{P}_{\Theta}^t$ on $\Omega_t$ and $\mathbb{P}_{\Theta}$ on $\Sn$
\begin{align}\label{eq_relationPtP}
\mathbb{P}_{\Theta}^t[(k_1, k_2, ...) | v=n] = \PT{C_1 = k_1, \dots, C_n = k_n}.
\end{align}
This gives in analogy to (2) in \cite{Ha90}
\begin{align}
\label{eq_extended_cycle_index}
e^{g_{\Theta}(t)} \ \mathbb{E}_{\Theta}^t[\Psi] = \sum_{n=1}^{\infty}{h_n \mathbb{E}_{\Theta}[\Psi_n] t^n} + \Psi(0),
\end{align}
where $\Psi$ is a function on the space $\Omega_t$ and $\Psi_n:\Sn \to \C$ is defined as $\Psi_n:= \Psi(C_1,C_2,\dots)$.
Again, we omit the proofs of \eqref{eq_relationPtP} and \eqref{eq_extended_cycle_index} since they are identical to those in \cite{Ha90}.

The identity \eqref{eq_extended_cycle_index} is true for all $0<t<r$ and thus remains valid as formal power series.
Hence, we get with \eqref{eq:def_Bn*}
\begin{align*}
 h_n E_{\Theta}^{B_n^*} =
\nth{ e^{g_{\Theta}(t)} \ \mathbb{E}_{\Theta}^t\left[\big(B^*(x) - B^*(x_1)\big)^2 \big(B^*(x_2) - B^*(x)\big)^2  \right]  }
\end{align*}
and
\begin{align*}
B^*(x)(k_1,k_2,\dots)
:= 
\frac{1}{\sqrt{\vartheta \log n}} \sum_{m=1}^{\lfloor n^x \rfloor} \left(k_m - \frac{\theta_m t^m}{m} \right).
\end{align*}
A small calculation using that the fact that the coordinates on $\Omega_t$ are independent Poisson distributed completes the proof of \eqref{eq:joint_second_moment}.

\end{proof}

\subsection{Restricted measure}
\label{sec:restriction-small}
In the last subsection we only considered the weak convergence of the process $\widetilde{B}_n(.)$ without restriction of the probability measure, meaning under the condition $A_n = \set{1, \dots, n}$. Verifying the proof in Subsection~\ref{sec:funct_limit_without_restrictions} carefully, one notices that our argumentation is based on the equations \eqref{eq:char_increments} and \eqref{eq:joint_second_moment}, but they require only minor modifications in case $A_n \neq \set{1, \dots, n}$.
Thus, one can apply the proof of Subsection~\ref{sec:funct_limit_without_restrictions} for many possible restrictions $A_n$ (as long as the assumptions of Theorem~\ref{thm:total_cycles_asymp_with_restriction} are satisfied). 
Since the argumentation for all the interesting cases are similar we restrict the investigation to
$A_n = \set{\lceil n^a\rceil, \dots, n}$ with $0\leq a <1$. In this case, the characteristic function of $B_n(x)$ for $0 \leq x <1$ behaves like
\begin{align*}
 \ETA{A_n}{\exp(is B_n(x))} = \exp\left((e^{is}-1)L_{M_n}(r)\right) \left( 1 + O\left(\frac{\max\set{n^x, n^a}}{n}\right)\right),
\end{align*}
where $M_n = A_n \cap \set{1, ..., \lfloor n^x \rfloor }  = \set{\lceil n^a\rceil, \dots, \lfloor n^x \rfloor}$.
We get with \eqref{eq:theta_k_to_vartheta}
\begin{align*}
 L_{M_n}(r) = \sum_{m=1}^{\lfloor n^x \rfloor } \frac{\theta_m}{m}r^m = (x-a)\vartheta\log n \one_{\set{x \geq a}} +O(1).
\end{align*}
One  can now use the same argumentation as in Section~\ref{sec:funct_limit_without_restrictions} to show that
\begin{align}\label{eq:Bn with restriction}
 \frac{B_n(x) - \max\set{x-a,0}\vartheta \log n }{\sqrt{\vartheta\log n }} 
\overset{d}{\longrightarrow}
\mathcal{W}_a(x),
\end{align}
where $\mathcal{W}_a(x)$ is the continuous process on $[0,1]$ with 
\begin{align*}
\mathcal{W}_a(x) \stackrel{d}{=} 
\begin{cases}
    \mathcal{N}(0,x-a) & \text{if } x\geq a,\\
0 & \text{otherwise.}                 
\end{cases}
\end{align*}
In other terms, for $A_n = \set{\lceil n^a\rceil, \dots, n}$, the process defined on the left-hand side of (\ref{eq:Bn with restriction}) converges weakly to a Brownian motion started at $x = a$.

\subsection{Restriction to even and odd cycles}
\label{sec:restriction_even_odd}
This subsection is devoted to the asymptotic behaviour of the processes
\begin{align}
 B_n^{(ev)}(x):= \sum_{\substack{1 \leq m \leq n^{x}\\m \text{ even}}} C_m 
 \quad \text{ and } \quad
 B_n^{(odd)}(x):= \sum_{\substack{1 \leq m \leq n^{x}\\m \text{ odd}}} C_m. 
\end{align}
For simplicity we assume that we have no restrictions to the cycle lengths, that is $A_n =\set{1,\dots,n}$ in Definition~\ref{def_A-weighted_probabililty_measure}. As for the process $B_n(.)$ in Subsection~\ref{sec:funct_limit_without_restrictions}, we will find an appropriate rescaling for $B_n^{(ev)}(.)$ and $B_n^{(odd)}(.)$ in order to prove joint convergence to the Brownian motion, see Theorem~\ref{thm:ev_odd_BM}.

First, we need to compute the characteristic function. For $0 \leq x_1,x_2 \leq 1$ we have
\begin{align}
\label{eq:char_even_odd}
&h_n \ET{\exp\left(is_1 B_n^{(ev)}(x_1) +is_2 B_n^{(odd)}(x_2) \right)}\\
=
&\nth{\exp\left( g_\Theta(t) + (e^{is_1}-1) L_{D_n^{(ev)}}(t) +(e^{is_2}-1) L_{D_n^{(odd)}}(t) \right)  }\nonumber
\end{align}
with $D_n^{(ev)} = \set{m \leq n^{x_1}| m \text{ even}}$ and $D_n^{(odd)} = \set{m \leq n^{x_1}| m \text{ odd}}$.
This is proven by our usual argumentation. We now can apply Theorem~\ref{thm:total_cycles_asymp_with_restriction} for $0 \leq x_1,x_2 < 1$ and get
\begin{align*}
 &\ET{\exp\left(is_1 B_n^{(ev)}(x_1) +is_2 B_n^{(odd)}(x_2) \right)} \\
&= 
\exp\left((e^{is_1}-1)L_{D_n^{(ev)}}(r)\right) \exp\left((e^{is_2}-1)L_{D_n^{(odd)}}(r)\right) \left( 1 + O\left(\frac{1}{n}\right)\right).
\end{align*}
With \eqref{eq:theta_k_to_vartheta} follows that
\begin{align}
 L_{D_n^{(ev)}}(r) =  x_1 \frac{\vartheta}{2} \log n + O(1)
\ \ \text{ and } \ \
 L_{D_n^{(odd)}}(r) =  x_2 \frac{\vartheta}{2} \log n + O(1)
\end{align}
and therefore we define the rescaled processes
\begin{align}
 \widetilde{B}_n^{(ev)} (x) :=  \frac{B_n^{(ev)}(x) - x \frac{\vartheta}{2} \log n}{\sqrt{\frac{\vartheta}{2}  \log n}}
\text{ and }
\widetilde{B}_n^{(odd)} (x) :=  \frac{B_n^{(odd)}(x) - x \frac{\vartheta}{2} \log n}{\sqrt{\frac{\vartheta}{2}  \log n}}.
\end{align}
Our aim is to prove the following theorem.
\begin{theorem}\label{thm:ev_odd_BM}
 The processes $\widetilde{B}_n^{(ev)} (x)$ and $\widetilde{B}_n^{(odd)} (x)$ converge, as 
$n \rightarrow \infty$, to two independent standard Brownian motions for $0 \leq x \leq 1$.
\end{theorem}

\begin{proof}
Using the same argumentation as in the proof of Corollary~\ref{cor_cltBn}, we see that for $0 \leq x_1,x_2 < 1$
\begin{align}
 \left(\widetilde{B}_n^{(odd)} (x_1),\widetilde{B}_n^{(odd)} (x_2)   \right)
\stackrel{d}{\longrightarrow} \left(\mathcal{N}_1,\mathcal{N}_2 \right),
\end{align}
where $\mathcal{N}_1$ and $\mathcal{N}_2$ are independent centered Gaussian random variables with variance $x_1, x_2$, respectively. More interesting and difficult is the behaviour for $x_1=1$ and/or $x_2=1$.
We have for $x_1=1$ 
\begin{align*}
 L_{D_n^{(ev)}}(t) = \sum_{\substack{1 \leq m \leq n\\m \text{ even}}} \frac{\theta_m}{m} t^m 
=\frac{1}{2} g_\Theta(t) +  \frac{1}{2} g_\Theta(-t) -t^{n+1} f_n(t) 
\end{align*}
where $f_n(t)$ is a holomorphic function around $0$.
We thus get
\begin{align}
\label{eq:even_gen_at_one}
h_n \ET{\exp\left(is_1 B_n^{(ev)}(1) \right)}
=
\nth{\exp\left( g_\Theta(t) + \frac{e^{is_1}-1}{2}\left(g_\Theta(t) + g_\Theta(-t)\right)  \right)  }.
\end{align}
By an analogue argumentation for $x_1=x_2=1$ we obtain
\begin{align}
\label{eq:odd_gen_at_one}
&h_n \ET{\exp\left(is_1 B_n^{(ev)}(1) + is_2 B_n^{(odd)}(1) \right)}\nonumber\\
=
&\nth{\exp\left(  g_\Theta(t) + \frac{e^{is_1}-1}{2}(g_\Theta(t) + g_\Theta(-t)) + \frac{e^{is_2}-1}{2}(g_\Theta(t) - g_\Theta(-t))  \right)  }\nonumber\\
=
&\nth{\exp\left( \frac{e^{is_1} +e^{is_2}}{2} g_\Theta(t) + \frac{e^{is_1} -e^{is_2}}{2} g_\Theta(-t)   \right)  }.
\end{align}
We cannot apply Theorem~\ref{thm:total_cycles_asymp_with_restriction} for \eqref{eq:even_gen_at_one} and \eqref{eq:odd_gen_at_one}  since 
the functions in this equations have singularities at the points $r$ and $-r$.
However, a modification of Theorem~\ref{thm:total_cycles_asymp_with_restriction} applies to this situation, one simply has to replace the curve $\gamma$ in Figure~\ref{fig:curve_flajolet_4} by the curve in Figure~\ref{fig:curve_for_two_sing}. 

\begin{figure}[h]
 \centering
 \includegraphics[width=0.35 \textwidth]{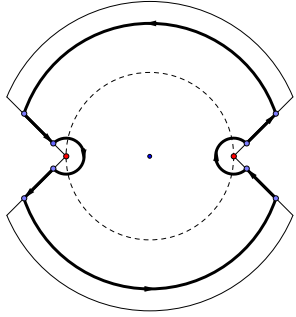}
 \caption{The curve $\gamma$ for two singularities}
 \label{fig:curve_for_two_sing}
\end{figure}

This gives Theorem~\ref{thm:total_cycles_asymp_with_restriction_2}, see below, which we apply for (\ref{eq:odd_gen_at_one}) if $|s_1|, |s_2| \leq \pi/4$ and obtain
\begin{align*}
\ET{\exp\left(is_1 B_n^{(ev)}(1) + is_2 B_n^{(odd)}(1) \right)}
=n^{\left(\frac{e^{is_1}+e^{is_2}}{2}-1\right)\vartheta} \left( F(s_1,s_2) + \frac{\bar{d}_n}{n}\right)
\end{align*}
with $F(s_1,s_2)$ a holomorphic function in a neighbourhood of the origin with $F(0,0)=1$. Then, the same computation as in the proof of Corollary~\ref{cor_cltBn} shows that 
\begin{align}
 \left(\widetilde{B}_n^{(odd)} (1),\widetilde{B}_n^{(odd)} (1)   \right)
\stackrel{d}{\longrightarrow} \left(\mathcal{N}_1,\mathcal{N}_2 \right)
\end{align}
with $\mathcal{N}_1,\mathcal{N}_2$ two independent standard Gaussian random variables.

Finally, it remains to prove that the increments of $\widetilde{B}_n^{(ev)} (.)$ and $\widetilde{B}_n^{(odd)} (.)$ are independent and the tightness of both processes.
This argumentations are completely similar to those in Section~\ref{sec:funct_limit_without_restrictions}, see \eqref{eq:char_increments} and \eqref{eq:joint_second_moment}, and we thus omit them. 
\end{proof}

\begin{theorem}
\label{thm:total_cycles_asymp_with_restriction_2}
Let $g(t)$ in $\mathcal{F}(r,\vartheta, K)$ be given and $D_n^{(j)}, d_n^{(j)}$ and $\bar{d}_n$ be as in Theorem~\ref{thm:total_cycles_asymp_with_restriction}. We define 
\begin{align*}
G_n(t,w_1,w_2,v_1,...,v_k) := \exp \left( w_1 g(t) +w_2 g(-t)+ \sum_{j=1}^k{v_j L_{D_n^{(j)}}(t)}\right) 
\end{align*}
with $w_1,w_2, v_1,\dots,v_k\in\C$ and $L_{D_n^{(j)}}(t)$ as in Lemma~\ref{lem:generating_hn_A}. Suppose further that $\bar{d}_n$ fulfils the assumption \eqref{eq:assumption_on_dn}
We then have for each $b\in\N$ fix  
\begin{align}
\label{eq_asymp_Gn_2}
&[t^{n-b}]\left[ G_n(t,w_1,w_2,v_1,...,v_k) \right] \\
=& 
\frac{e^{Kw_1 } n^{w_1 \vartheta -1}e^{w_2 g(-r)}}{r^{n-b}}   
\exp\Bigg( \sum_{j=1}^k{v_j L_{D_n^{(j)}}(r)} \Bigg)  
\Bigg(\frac{1}{\Gamma(w_1 \vartheta)} +  O\bigg(\frac{\bar{d}_n}{n} \bigg) \Bigg) \nonumber\\
&+ O\left( n^{\max\{\Re(w_2),0\}-1} \right)\nonumber
\end{align}
uniformly for bounded $|w_1|,|w_2|, |v_1| ,..., |v_k|$ and $\Re(w_1)\geq 0$.
\end{theorem}

\bibliography{literatur}
\bibliographystyle{abbrv}

\end{document}